\documentclass[a4paper,12pt]{article}
\usepackage{amsmath,amsthm,amssymb,amsfonts} 
\usepackage{bbm,bm} 
\usepackage{latexsym} 
\usepackage{mathrsfs} 
\usepackage{dsfont}
\usepackage{threeparttable} 
\usepackage{tabularx} 
\usepackage{booktabs} 
\usepackage{multirow} 
\usepackage{pifont}
\usepackage{graphicx,subfigure} 
\usepackage{float}
\usepackage{color} 
\usepackage{indentfirst} 
\usepackage{geometry}
\renewcommand{\baselinestretch}{1.5} 
\usepackage{caption}
\usepackage{lineno} 
\usepackage{microtype}
\usepackage[numbers,sort&compress]{natbib} 
\usepackage[colorlinks,
            linkcolor=blue,
            citecolor=red
            ]{hyperref} 
\usepackage{epstopdf} 

\newtheorem{thm}{Theorem}[section]

\newtheorem{lem}[thm]{Lemma}

\newtheorem{que}[thm]{Question}
\newtheorem{proper}[thm]{Property}
\begin{document}

\setlength{\baselineskip}{20pt}

\begin{center}

{\Large \bf 
The number of perfect matchings in 3-connected planar graphs 
$^{\text{\ding{73}}}$}
\vspace{4mm}

{Wuxian Chen, Xinyu Dai, Fuliang Lu$^{\dagger}$}

\vspace{4mm}

\footnotesize{School of Mathematics and Statistics, Minnan Normal University,
Zhangzhou, 363000, PR China}

\renewcommand\thefootnote{}
\renewcommand{\baselinestretch}{1.2}
\footnote{$^{\text{\ding{73}}}$ This work is supported by NSFC\,(Grant No. 12271235),  Fujian Key Laboratory of Granular Computing and Applications (Minnan Normal University), Natural Science Foundation of Fujian Province (No. 2026J002034) and Institute of Meteorological Big Data-Digital Fujian.}

\renewcommand{\baselinestretch}{1.2}
\footnote{$^{\dagger}$ Corresponding author. E-mail addresses:
wuxian.chen@qq.com (W. Chen), daixinyu03520@163.com (X. Dai) and flianglu@163.com (F. Lu).}

\end{center}
\

\noindent {\bf Abstract}:
A graph is matchable if it admits a perfect matching. Recently, Goedgebeur et al. 
asked whether there exists a constant $c<12$ such that  infinitely many matchable planar $3$-connected graphs, each with exactly $c$ perfect matchings. We answer this question by proving that every matchable planar $3$-connected graph on at least 40 vertices has at least 12 perfect matchings, and this lower bound is sharp. 
\vspace{2mm}

\vspace{2mm}


{\setcounter{section}{0}
\section{\normalsize Introduction}\setcounter{equation}{0}

All graphs considered in this paper are finite and contain no loops. We follow \cite{BM} for undefined notation and terminology. 
A {\em matching} of a graph is a set of pairwise nonadjacent edges. 
A \emph{perfect matching} $M$ of a graph is a matching such that every vertex is incident with exactly one edge of $M$.  A graph is \emph{matchable} if it has a perfect matching. Research concerning the number of perfect matchings in graphs, including in planar graphs, has been extensive. In the context of statistical mechanics (perfect matchings correspond to dimer coverings of lattices), Kasteleyn \cite{Kasteleyn} developed a polynomial-time algorithm to compute the number of perfect matchings for arbitrary planar graphs in the 1960s. Zaks \cite{Zaks} proved that every $k$-connected simple matchable graph has at least $k!!$ perfect matchings, and this lower bound is sharp for odd $k$. Later, Lov\'{a}sz \cite{L72} proved that every $k$-connected matchable simple graph that is not bicritical has at least $k!$ perfect matchings, where a graph on at least 4 vertices is \emph{bicriticcal} if the removal of any two distinct vertices results in a matchable graph. Chudnovsky and Seymour \cite{CS} showed that all bridgeless cubic planar graphs have exponentially many perfect matchings in terms of the number of vertices. This result was extended to arbitrary bridgeless cubic graphs by Esperet et al. \cite{EKKKN}, which was conjectured by Lov\'{a}sz and Plummer. More recently, Goedgebeur et al. \cite{GJVZ} investigated the minimum non-zero number of perfect matchings in planar graphs. They proved that there is a positive constant number of perfect matchings for $2$-connected simple planar graphs of minimum degree $3$ and $3$-connected planar graphs. Furthermore, they raised the following question. 

\begin{que}[\cite{GJVZ}]\label{question}
Is there a constant $c<12$ such that there exist infinitely many matchable planar $3$-connected graphs, each with exactly $c$ perfect matchings ?
\end{que}

In this paper, we answer the above question by proving the following main result.  

\begin{thm}\label{main}
Let $n\geq 40$ be an even integer. Then every matchable planar $3$-connected graph on $n$ vertices has at least 12 perfect matchings. 
\end{thm}

Moreover, the lower bound in Theorem \ref{main} is sharp: for every even integer $n\geq 10$, Goedgebeur et al. \cite{GJVZ} constructed a simple planar 3-connected graph with $n$ vertices and precisely 12 perfect matchings. 

In the next section, we will give some notions and useful results.  
In Section 3, we give the proof of Theorem \ref{main}.

\section{\normalsize Preliminaries}

\subsection{Basic notions and results}
Let $G$ be a graph with vertex set $V(G)$ and edge set $E(G)$. If $G$ has multiple edges, then we write $G_{\text{sim}}$ for the underlying simple graph of $G$. 
The \emph{degree} of a vertex $v$ in $G$, denoted by $d_G(v)$, is the number of edges of $G$ incident with $v$. A \emph{k-degree vertex} is a vertex of degree $k$. 
Let $\delta(G)$ be the minimum degree of $G$. For a set $S\subseteq V(G)$, denote by $N_G(S)$ the set of all vertices in $V(G)\setminus S$ adjacent to a vertex of $S$. If $|S|=1$, say $S=\{u\}$, then $N_G(u):=N_G(\{u\})$ denotes the set of neighbors of $u$ in $G$. For two vertex subsets $X,Y\subseteq V(G)$, denote by $E_G(X,Y)$ the set of all edges of $G$ with one end in $X$ and the other in $Y$. 

An edge of a graph is \emph{allowed} if it  lies in a perfect matching, and \emph{forbidden} otherwise. For a matchable graph $G$, denote by $\Phi(G)$ the number of perfect matchings in $G$. A useful lower bound on $\Phi(G)$ is given by the following lemma. 

\begin{lem}[Consequence of Corollary 5.13 in \cite{EPL}]\label{formula}
Let $G$ be a matchable graph on at least 4 vertices. Then
$$\Phi(G)\geq\frac{|E(F(G))|-|V(G)|}{2}+2,$$ 
where $F(G)$ denotes the subgraph of $G$ induced by all allowed edges.
\end{lem}

A connected graph with at least two vertices is \emph{matching covered} if each edge is allowed. An edge $e$ of a matching covered graph $G$ is {\em removable} if $G-e$ is matching covered, and nonremovable otherwise. 

\begin{lem}[\cite{HLZ}]\label{removable}
Let $G$ be a matching covered bipartite graph on at least $4$ vertices.
If $\delta(G)\geq3$, then the subgraph induced by all nonremovable edges of $G$ is a forest (a graph that contains no cycles).
\end{lem}

\begin{lem}\label{3-PMs}
Let $G$ be a matching covered graph. If $\delta(G) \geq 3$ and $\Phi(G) = 3$, then $G \cong K_4$ or $G$ is $K_2$ with exactly three multiple edges, where $K_n$ is a complete graph on $n$ vertices. 
\end{lem}
\begin{proof}
Since $\delta(G)\geq 3$ and $G$ is matching covered, $G$ is not induced by a cycle. By Corollary 7 in \cite{CLM14}, there exists a nonempty edge subset $R$ with $|R|\leq 2$ such that $G-R$ remains matching covered. If the maximum degree of $G-R$ is less than 3, then $|V(G)|\leq 4$. Since $\Phi(G)=3$, $G$ is either $K_2$ with exactly three edges or $G\cong K_4$. If the maximum degree of $G-R$ is 3, then $\Phi(G-R)\geq 3$. Since $G$ is matching covered, $G$ has a perfect matching containing an edge of $R$, which is distinct from any perfect matching of $G-R$. So $\Phi(G)>\Phi(G-R)\geq 3$, a contradiction. Therefore, the lemma holds.  
\end{proof}

A component of a graph is called \emph{odd}  if it has an odd number of vertices. Denote by $o(G)$ the number of odd components of a graph $G$. 
The \emph{deficiency} of a graph $G$, denoted by $\operatorname{def}(G)$, is defined as $\operatorname{def}(G) =\max\big\{o(G-S)-|S|\,\big|\, S\subseteq V(G)\big\}$. If $G$ is matchable, then $\operatorname{def}(G)=0$. 
A \emph{barrier} of a graph $G$ is a set $S\subseteq V(G)$ such that $o(G-S)=|S|+\operatorname{def}(G)$. Moreover, a barrier $S$ of $G$ is \emph{nontrivial} if $|S|\geq2$. 
For a barrier $S$ of a non-bipartite graph $G$, denote by $G(S)$ the bipartite graph obtained from $G$ by deleting edges with both ends in $S$, 
and then shrinking each odd component of $G-S$ to a single vertex. If $G$ is a bipartite  graph with the bipartition $(S,T)$, then we write $G=G(S)=G(T)$.

A graph $G$ is {\em factor-critical} if $G-u$ is matchable for any $u\in V(G)$. 

\begin{lem}[\cite{LP09}]\label{p-matchable}
Let $G$ be a matchable graph. Then the following two statements hold. 

{\rm (i)} $G$ is bicritical if and only if $G$ has no nontrivial barriers,

{\rm (ii)} If $S$ is a maximal barrier of $G$, then each component of $G-S$ is factor-critical.  
\end{lem}

Recall the core notion in the Gallai-Edmonds structure theorem (see \cite{LP09} for details): $D(G)$ collects all vertices in a graph $G$ which are not incident with any edge of at least one maximum matching of $G$, and $A(G)$ consists of vertices in $V(G)\setminus D(G)$ adjacent to at least one vertex in $D(G)$. 

\begin{lem}[\cite{LP09}]\label{A(G)-barrier}
Let $G$ be any graph. Then $A(G)$ is a barrier of $G$. 
\end{lem}

A graph is \emph{elementary} if its allowed  edges formed a connected (spanning) subgraph. Let $S_1,\ldots,S_k$ be all the maximal barriers of a graph $G$. We define $\mathcal{P}(G)=\{S_1,\ldots,S_k\}$. Then $\mathcal{P}(G)$ is a partition of $V(G)$ if $G$ is elementary (see \cite[Lemma 5.2.1]{LP09}).  
Next we introduce several helpful results of elementary graphs. 

\begin{lem}[\cite{LP09}]\label{saturated-ele}
Let $G$ be an elementary graph, $S\in \mathcal{P}(G)$, and $|S|\geq 2$. Then $G(S)$ is an elementary graph. 
\end{lem}

\begin{lem}[\cite{LP09}]\label{subdivide-barrier}
Let $G$ be an elementary graph. Then the following statements hold. 

{\rm (i)} For an allowed edge $e=xy$ in $G$, if $G'$ is obtained by subdividing $e$ by the insertion of two new vertices $u$ and $v$, then the classes of $\mathcal{P}(G')$ are the same as those of $G$ except $u$ joins the class of $y$ and $v$ joins the class of $x$,

{\rm (ii)} If $e\notin E(G)$, then $e$ is allowed in $G+e$ if and only if $x$ and $y$ lie in different classes of $\mathcal{P}(G)$. 

\end{lem}

\begin{thm}[\cite{Hetyei}]\label{p-ele}
A bipartite graph is elementary if and only if it is matching covered. 
\end{thm}


A simple matchable graph $G$ is said to be \emph{saturated} if $\Phi(G+e)>\Phi(G)$ for all edges not in $G$. For saturated non-elementary graphs, we have the following famous Cathedral  theorem. 

\noindent \textbf{The Cathedral Construction.} Let $G_0$ be any saturated elementary graph. To each class $S \in \mathcal{P}(G_0)$ assign an (already constructed) saturated graph $G_S$ or the empty set. For each $S \in \mathcal{P}(G_0)$, join every vertex of $S$ to every vertex of $G_S$. When $G_S$ is not empty,  we call $G_S$ the \emph{tower} over $S$. 
(Note that $G_S$ and $S$ are vertex-disjoint.)

\begin{thm}[The Cathedral Theorem, \cite{LP09}]\label{Cathedral}
{\rm (a)} Every graph $G$, built up by iterating the Cathedral Construction using smaller saturated graphs, is itself saturated.

{\rm (b)} The allowed edges of $G$ are precisely those edges which are allowed in one of the elementary graphs used in one of the steps.

{\rm (c)} Conversely, if $G$ is any saturated graph, it can be built up using the Cathedral Construction starting with a saturated elementary graph $G_0$ and a collection of $|\mathcal{P}(G_0)|$ smaller saturated graphs (some perhaps empty) already constructed. 
\end{thm}

\begin{lem}\label{p-saturated-Cathedral}
Let $G$ be a saturated graph, and $G_0$ be a saturated elementary graph such that $G$ is built up using the Cathedral Construction starting with $G_0$. For a set $S\in \mathcal{P}(G_0)$, 

{\rm (i)}(\cite{LP09}) $A(G-x)\cup \{x\}=S$ for each vertex $x\in S$,

{\rm (ii)} $S$ is also a barrier of $G$.
\end{lem}
\begin{proof}
We prove (ii). Let $x\in S$, and $H=G-x$. Assume that $M$ is a perfect matching of $G$. Then there is a vertex $y$ of $G$ so that $xy\in M$. Since each vertex of $H$ other than $y$ is incident with exactly one edge of $M\setminus \{xy\}$,  $\operatorname{def}(H) = \max\big\{ o(H-S) - |S| \,\big|\, S \subseteq V(H)\}=1$. Since $A(H)$ is a barrier of $H$ by Lemma \ref{A(G)-barrier}, $o(H-A(H))-|A(H)|=1$. 
By (i), $o(G-S)=o(G-x-A(H))=o(H-A(H))=1+|A(H)|=1+|S\setminus \{x\}|=|S|$. So $S$ is a barrier of $G$, and (ii) holds.
\end{proof}

\subsection{Bisubdivision and bicontraction }

To \emph{subdivide} an edge $e=uv$ is to delete $e$, add a new vertex $x$, and add edges $ux$ and $xv$, that is, $uv$ becomes the path $uxv$. 
Any graph derived from a graph $G$ by a sequence of edge subdivisions is called 
a \emph{subdivision} of $G$. An \emph{even subdivision} of $G$ is a subdivision in which an even number of vertices are inserted on each edge of $G$.

\begin{proper}[\cite{BM}]\label{p-planar}
Let $G$ be a graph. Then 

{\rm (i)} $G$ is planar if and only if every subdivision of $G$ is planar, 

{\rm (ii)} If $G$ is planar, then any graph obtained from $G$ by a sequence of vertex and edge deletions and edge contractions is also planar. 
\end{proper}

The \emph{bicontraction of a vertex $v$} of degree two, with precisely two different neighbors $v_1$ and $v_2$, in a graph  consists of shrinking the set $\{v,v_1,v_2\}$ to a single vertex. 
The \emph{retract} of a matching covered graph $G$, denoted by $\widehat{G}$, is the graph obtained from $G$ by the following recursive procedure: if $|V(G)|=2$ or $\delta(G)\geq 3$, then $\widehat{G}:=G$, and otherwise $\widehat{G}:=\widehat{G_v}$, where $G_v$ is the graph obtained by bicontracting 2-degree vertex $v$ in $G$. The retract of a matching covered graph is unique up to isomorphism. Clearly, $\widehat{G}$ is also matching covered.  From the definition, we easily obtain the following two properties.

\begin{proper}\label{bicontract}
Let $G$ be a graph. Then 

{\rm (i)} $\Phi(G)=\Phi(\widehat{G})$, 

{\rm (ii)} For an edge $e=uv$ of $G$, $\Phi(G)=\Phi(\widehat{G-e})+\Phi(G-\{u,v\})$.
\end{proper}

Denote by $\mathcal{F}_n$ the class of all bipartite matching covered graphs $G$ with $n$ vertices and $\delta(G)\geq 3$. 

\begin{proper}[\cite{CLM}]\label{small}
Let $G\in \mathcal{F}_n$. Then

{\rm (i)} $\Phi(G)\geq 12$ when $n\in \{10,12\}$,

{\rm (ii)} If $n=8$, then $\Phi(G)\geq 12$ unless 
$G\in \{B_8,B_8^+, P_4\oplus K_{3,3}\}$ (see Fig. \ref{small-graph}).
\end{proper}
\begin{figure}
\centering
\includegraphics{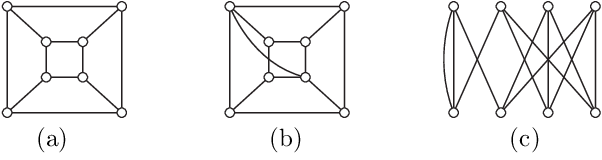}
\caption{\label{small-graph} (a)\ $B_8$, (b)\ $B_8^+$, (c)\ $P_4\oplus K_{3,3}$.}
\end{figure}

\begin{proper}\label{p}
Let $G\in\mathcal{F}_n$, $n\geq 14$. Then $\Phi(G)\geq 12$.  
\end{proper}
\begin{proof}
Since $G$ is matching covered and $G\neq K_2$, $G$ contains a cycle. By Lemma \ref{removable}, $G$ has a removable edge, say $e$. Then $G':=G-e$ is also matching covered. Since $\delta(G)\geq 3$, the possible 2-degree vertices of $G'$ can only be the two ends of $e$. So  $\widehat{G'}\in \mathcal{F}_n$ ($n\geq 10$). If $|V(\widehat{G'})|\in \{10,12\}$, then $\Phi(\widehat{G'})\geq 12$ by Property  \ref{small}(i).   
If $|V(\widehat{G'})|\geq 14$, by using the above argument and
Property \ref{bicontract}(ii) repeatedly, we obtain $\Phi(\widehat{G'})\geq min\{f(n-2),f(n-4)\}\geq min\{f(n-4),f(n-6)\}\geq \cdots \geq min\{f(12),f(10)\}\geq 12$, where $f(n)=min\{\Phi(G): G\in \mathcal{F}_n\}$. By Property \ref{bicontract}(ii), we have $\Phi(G)\geq 12$. 
\end{proof}

\section{\normalsize Proof of Theorem \ref{main}}
We only need to consider simple graphs. Let $G$ be a matchable simple graph on $n$ vertices. Let $G^*$ be the saturated graph containing $G$ as a spanning subgraph, which is obtained by adding edges to $G$ as long as no new perfect matchings are generated. Then each edge of $G^*$ in $E(G^*)\setminus E(G)$ is forbidden, and so $\Phi(G)=\Phi(G^*)$. If $G^*$ is non-elementary, then $G^*$ can be built up using the Cathedral Construction starting with a saturated elementary graph $G^*_0$ by Theorem \ref{Cathedral}(c). Otherwise, we set $G^*_0=G^*$. 

If $G^*_0$ is bicritical, then each class in $\mathcal{P}(G^*_0)$ is a singleton by Lemma \ref{p-matchable}(i). Since $G$ is 3-connected, $G^*$ is also 3-connected. Together with the Cathedral Construction, we obtain that the tower $G^*_S$ over $S$ is empty for each set $S\in \mathcal{P}(G^*_0)$. So $G^*_0=G^*$. Since $G^*_0$ is bicritical, each edge of $G^*_0$ is allowed. So $G^*_0=G^*=G$. 
By Lemma~\ref{formula}, $$\Phi(G)\geq\frac{|E(F(G))|-n}{2}+2=\frac{|E(G)|-n}{2}+2.$$ Since $G$ is 3-connected, $\delta(G)\geq3$.
Then $2|E(G)|\geq 3n$, which gives $\Phi(G)\geq\frac{n}{4}+2$. As $n\geq 40$, $\Phi(G)\geq 12$, which proves the theorem. 

In the following, suppose that $G^*_0$ is non-bicritical. Then $G^*_0$ has a nontrivial barrier by Lemma \ref{p-matchable}(i). 
Let $S$ be a nontrivial barrier in $\mathcal{P}(G^*_0)$. Together with the fact that $G^*_0$ is elementary, we have that $H:=G^*_0(S)$ is elementary by Lemma \ref{saturated-ele}. 
Further, $H$ is matching covered by Theorem \ref{p-ele}.  
Since each component of $G^*_0-S$ is factor-critical by Lemma \ref{p-matchable}(ii), each edge of $G^*_0$ corresponding to an edge of $H$ is allowed. So each perfect matching of $H$ can be extended to a perfect matching of $G^*_0$. By Theorem \ref{Cathedral}(b), each edge of $H$ corresponds to an allowed edge of $G$, and each perfect matching of $H$ can be extended to a perfect matching of $G^*$. It follows that $\Phi(G^*)=\Phi(G)\geq\Phi(H)$. We will show that $\Phi(H)\geq 12$. 

Since $G$ is planar, $H$ is also planar by  Property \ref{p-planar}(ii).   
By Lemma \ref{p-saturated-Cathedral}(ii), $S$ is also a nontrivial barrier of $G^*$. Since $o(G^*-S)=|S|\geq 2$, $G^*-S$ is disconnected. Further, since $G^*$ is 3-connected, we have $|S|\geq 3$. Then $|V(H)|\geq 6$. 
Assume that $S=\{a_1,a_2,\ldots,a_k\}$ and $T=V(H)\setminus S=\{b_1,b_2,\ldots,b_k\}$, where $k\geq 3$.

\vspace{8pt}\noindent
{\textbf{Claim 1.}} For any $u\in T$, $d_H(u)\geq |N_H(u)|\geq 3$.

\begin{proof}
Let $K$ be a component of $G^*_0-S$. Since $S$ is a barrier of $G^*$ (Lemma \ref{p-saturated-Cathedral}(ii)), $K$ is contained in some component of $G^*-S$, say $K^*$. Since $G^*$ is 3-connected, $|N_{G^*}(V(K^*))|\geq 3$. By the Cathedral Construction, $N_{G^*}(V(K^*))=N_{G^*_0}(V(K))$. As $N_H(u)=N_{G^*_0}(V(K))$,  $d_H(u)\geq|N_H(u)|\geq 3$. 
\end{proof}

\noindent{\textbf{Claim 2.}} Let $b\in V(\widehat{H})\setminus V(H)$. 
If $b$ has exactly two neighbors, say $a_1$ and $a_2$, then $|E_{\widehat{H}}(\{b\},\{a_i\})|\geq 2$ for $i=1,2$.
\begin{proof}
Suppose to the contrary that $|E_{\widehat{H}}(\{b\},\{a_1\})|\leq 1$. 
Let $S'$ be the set of all 2-degree vertices in $H$ that are bicontracted to form $b$, and $T'=N_H(S')$. By Claim 1, $S'\subseteq S$. Then $|S'|+1=|T'|$, and  
$|E_H(S',T')|=2|S'|=2(|T'|-1)$. 
On the other hand, Claim 1 implies that at least $|T'|-1$ vertices in $T'$ have at least two neighbors in $S'$, while the remaining vertex of $T'$ has at least one neighbor in $S'$. So $|E_H(S',T')|\geq 2(|T'|-1)+1=2|T'|-1$, a contradiction. Hence Claim 2 holds. 
\end{proof}

Next we divide the proof into three cases based on $|V(H)|$. 

{\textbf{Case 1.}} $|V(H)|\leq 8$.

If $|V(H)|=6$, then $H_{\text{sim}}\cong K_{3,3}$ by Claim 1, where $K_{3,3}$ is the  complete bipartite graph on 6 vertices. So  $H$ is nonplanar, a contradiction. Thus  $|V(H)|=8$. If $\delta(H)=2$, then all 2-degree vertices of $H$ are in $S$ by Claim 1. Let $a_1$ be a 2-degree vertex of $H$, and $N_H(a_1)=\{b_1,b_2\}$. By Claim 1, $N_H(b_3)=N_H(b_4)=\{a_2,a_3,a_4\}$. If $\{a_2,a_3,a_4\}\subseteq N_H(\{b_1,b_2\})$, then $\widehat{H}_{\text{sim}}\cong K_{3,3}$, and so $\widehat{H}_{\text{sim}}$ is nonplanar, a contradiction to Property  \ref{p-planar}(ii).   
So $\{a_2,a_3,a_4\}\nsubseteq N_H(\{b_1,b_2\})$. By Claim 1, we may assume that $N_H(b_1)=N_H(b_2)=\{a_1,a_2,a_3\}$ (see Fig. \ref{8}(a)). Since $\delta(G)\geq 3$, $a_1$ is adjacent to a vertex in $S\setminus\{a_1\}$ or in the tower $G^*_S$ over $S$ by  the Cathedral Construction. If $a_1a_4\in E(G)$, then $H+a_1a_4$ contains a subdivision of $K_{3,3}$ with bipartite sets $\{a_2,a_3,a_4\}$ and $\{a_1,b_3,b_4\}$ (see Fig. \ref{8}(b)). As $K_{3,3}$ is nonplanar, $H+a_1a_4$ is nonplanar by Property \ref{p-planar}. But, since $G$ is planar, $H+a_1a_4$ is  planar by Property \ref{p-planar}(ii), a contradiction.   
If $a_1$ is adjacent to some vertex in $G^*_S$, then $a_1$ and at least two vertices from $\{a_2,a_3,a_4\}$ each have a neighbor in the connected subgraph $G_S^*\cap G$, because $G$ is 3-connected.  
Then the graph, obtained from $H\cup (G^*_S\cap G)\cup E_G(S,V(G^*_S))$ by contracting $V(G^*_S)$ to a single vertex, is nonplanar since it contains a subdivision of $K_{3,3}$ (see Fig. \ref{8}(c) and (d)), contradicting Property \ref{p-planar}. 
So $\emptyset\neq N_G(a_1)\setminus \{b_1,b_2\}\subseteq\{a_2,a_3\}$. But thus $(H\cup E_G(\{a_1\},\{a_2,a_3\}))-\{a_2,a_3\}$ has exactly two components, one with vertex set $\{a_1,b_1,b_2\}$ and the other with vertex set $\{a_4,b_3,b_4\}$. This implies that $\{a_2,a_3\}$ is a 2-vertex cut of $G$, a contradiction. Hence $\delta(H)\geq 3$. 
Since $H$ is planar, either $H\cong B_8$ or $\Phi(H)\geq 12$ by Property \ref{small}(ii). So we assume that $H\cong B_8$. We will prove that $\Phi(G^*)\geq 12$. 

\begin{figure}
\centering
\includegraphics{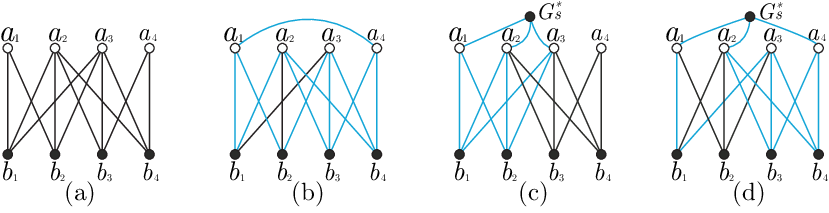}
\caption{\label{8} Illustration for the case that $|V(H)|=8$ and $\delta(H)=2$ in the proof of Theorem \ref{main}.}
\end{figure}

If $G^*_S$ is not empty, then the (planar) graph, obtained from $H\cup (G^*_S\cap G)\cup E_G(S,V(G^*_S))$ by contracting $V(G^*_S)$ to a single vertex, contains a subdivision of $K_{3,3}$ since $G$ is 3-connected, a contradiction to Property \ref{p-planar}. 
So $G^*_S$ is empty. Let $K$ be a component of $G^*_0-S$ with the largest number of vertices. If $K$ is trivial, then each component of $G_0^*-S$ is trivial. Since $G^*$ is 3-connected, $G_0^*=G^*$ by the Cathedral Construction. Since each component of $G_0^*-S$ is odd by Lemma \ref{p-matchable}(ii), $|V(G^*)|=n=8$, a contradiction. So $K$ is nontrivial. Assume that $G_1$ is the graph obtained from $G^*_0$ by contracting $V(G^*_0)\setminus V(K)$ to a single vertex $\overline{k}$. Then $d_{G_1}(\overline{k})=3$, and $|V(G_1)|\geq 4$. Since each edge of $H$ is allowed in $G^*$, each edge of $G_1$ incident with $\overline{k}$ is allowed. Since each edge of $B_8$ is contained in exactly 3 perfect matchings,  $\Phi(G^*)\geq \Phi(G_0^*)\geq 3\Phi(G_1)$ by Theorem \ref{Cathedral}(b). We then show that $\Phi(G_1)\geq 4$.  

Let $G_1'$ be the subgraph of $G_1$ formed by all allowed edges. Since $G_1$ is also elementary, $G_1'$ is matching covered. Then $\widehat{G_1'}$ is also matching covered,  and $\delta(\widehat{G_1'})\geq 3$. If $\Phi(\widehat{G_1'})=3$, then $\widehat{G_1'}\cong K_4$ or $\widehat{G_1'}$ is $K_2$ with exactly 3 multiple edges by Lemma \ref{3-PMs}. So either $G_1'\cong K_4$ or $G_1'$ is an even subdivision of $\widehat{G_1'}$. 
For the former, $G_1'=G_1\cong K_4$, which means that $K$ is a triangle. Then each  component of $G^*_0-S$ is a triangle or a singleton. So each class of $\mathcal{P}(G^*_0)\setminus S$ is a singleton. Since $G^*$ is 3-connected, we have that $G^*_0=G^*$ by the Cathedral Construction. 
It follows that $|V(G^*)|\leq 4+3\times4=16$, a contradiction. 
For the latter, each vertex of $\widehat{G_1'}$ constitutes a maximal barrier.  
By Lemma \ref{subdivide-barrier}(i), there is a class $S_1\in \mathcal{P}(G_1)$ such that $S_1$ contains a 2-degree vertex (say $w$) of $G_1'$, and $\overline{k}\notin S_1$. Let $K^*$ be the component of $G^*-S$ containing $K$, and $G_1^*$ be the graph obtained from $G^*$ by contracting $V(G^*)\setminus V(K^*)$ to a single vertex, also denoted by $\overline{k}$.  
As $d_{G}(w)\geq 3$, $w$ is incident with a forbidden edge of $G_1^*$, say $ww'$. 
If $w'\in V(K)$, then $w'\in S_1$ by Lemma \ref{subdivide-barrier}(ii). Tutte's theorem implies that $o(G_1^*-\{w,w'\})>2$. 
So $\{w,w'\}$ is a 2-vertex cut of $G_1^*$. As $\overline{k}\notin \{w,w'\}$, $\{w,w'\}$ 
is also a 2-vertex cut of $G^*$, a contradiction. If $w'\in V(K^*)\setminus V(K)$, then the same contradiction appears similarly. So $\Phi(\widehat{G_1})\geq 4$. By Property \ref{bicontract}(i), $\Phi(G_1)\geq 4$, which completes the proof in this case.

\begin{figure}
\centering
\includegraphics{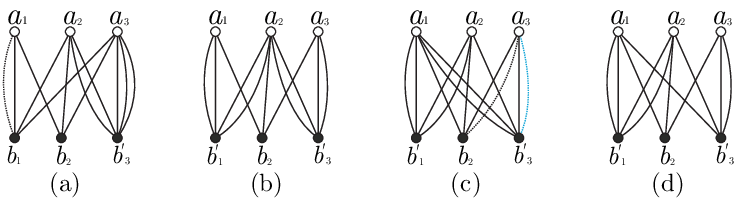}
\caption{\label{101} Illustration for the case that $|V(\widehat{H})|=6$ of Case 2 in the proof of Theorem \ref{main}.}
\end{figure}

{\textbf{Case 2.}} $|V(H)|=10$.

If $\delta(H)\geq 3$, then $\Phi(H)\geq 12$ by Property \ref{small}(i). So we assume that $\delta(H)=2$. Then all 2-degree vertices of $H$ belong to $S$ by Claim 1. 
By Property  \ref{p-planar}(ii), $\widehat{H}$ is planar. 
Claim 2 implies $|V(\widehat{H})|\geq 4$. If $|V(\widehat{H})|=8$, then the unique vertex of $\widehat{H}$, obtained by bicontracting one 2-degree vertex, has degree at least 4 by Claim 1. So $\Phi(H)\geq 12$ by Property \ref{small}(ii). If $|V(\widehat{H})|=6$, then $\widehat{H}$ contains either 1 or 2 vertices, obtained by bicontracting 2-degree vertices of $H$. For the former, assume that $V(\widehat{H})=\{a_1,a_2,a_3,b_1,b_2,b_3'\}$. By Claim 1,  $N_{\widehat{H}}(b_1)=N_{\widehat{H}}(b_2)=\{a_1,a_2,a_3\}$. Since $\widehat{H}$ is planar, we may assume that $N_{\widehat{H}}(b_3')=\{a_2,a_3\}$. Since $b_3'$ is obtained by bicontracting two 2-degree vertices of $H$, $d_{\widehat{H}}(b_3')\geq 5$ by Claim 1. By Claim 2, $|E_{\widehat{H}}(\{b_3'\},\{a_i\})|\geq 2$ and $|E_{\widehat{H}}(\{b_3'\},\{a_{5-i}\})|\geq 3$ for some $i\in \{2,3\}$, say $i=2$ (see Fig. \ref{101}(a)). 
As $\delta(\widehat{H})\geq 3$, $|E_{\widehat{H}}(\{a_1\},\{b_j\})|\geq 2$ for some $j\in \{1,2\}$. Then $\Phi(\widehat{H})\geq 2\times(3+2)+3+2=15$. 
For the latter, assume that $V(\widehat{H})=\{a_1,a_2,a_3,b_1',b_2,b_3'\}$. As $\widehat{H}$ is planar, $|N_{\widehat{H}}(b_i')\cap \{a_1,a_2,a_3\}|=2$ for some $i\in\{1,3\}$, say $N_{\widehat{H}}(b_1')=\{a_1,a_2\}$. By Claim 2, $|E_{\widehat{H}}(\{b_1'\},\{a_l\})|\geq 2$ for $l=1,2$. 
If $|N_{\widehat{H}}(b_3')\cap \{a_1,a_2,a_3\}|=2$, say $N_{\widehat{H}}(b_3')=\{a_2,a_3\}$ (see Fig. \ref{101}(b)). Claim 2 implies $\Phi(\widehat{H})\geq 2\times(2+2)+2\times2=12$. If $|N_{\widehat{H}}(b_3')\cap \{a_1,a_2,a_3\}|=3$, then $|E_{\widehat{H}}(\{b_3'\},\{a_i\})|\geq 2$ for some $i\in \{1,2,3\}$ as $d_{\widehat{H}}(b_3')\geq 4$ (see Fig. \ref{101}(c) and (d)). Similarly, we have $\Phi(\widehat{H})\geq min\{2\times3+2\times5,2\times3+2\times4,2\times3+2\times3\}=12$.

Now we assume that $|V(\widehat{H})|=4$. By Claim 1, 
$\widehat{H}$ is obtained from $H$ by bicontracting three 2-degree vertices in $S$, say $a_1$, $a_4$ and $a_5$. Denote by  $b_1'$ (resp. $b_2'$) the vertex of $\widehat{H}$ obtained by bicontracting 2-degree vertex $a_1$ (resp. vertices $a_4$ and $a_5$). 
Set $N_H(a_1)=\{b_1,b_2\}$ and $N_H(a_5)=\{b_4,b_5\}$. Then $N_H(b_3)=\{a_2,a_3,a_4\}$ by Claim 1. Further, we may assume that $\{a_3,a_4\}\subseteq N_H(b_5)$. Then $N_H(b_1)=N_H(b_2)=\{a_1,a_2,a_3\}$ and $N_H(b_4)=\{a_2,a_3,a_5\}$ by Claim 1 (see Fig. \ref{10}(a)). If $a_1a_5\in E(G)$, then $H+a_1a_5$ contains a subdivision of $K_{3,3}$ with bipartite sets $\{a_1,a_2,a_3\}$ and $\{b_1,b_2,b_5\}$ (see Fig. \ref{10}(b)), a contradiction to Property \ref{p-planar}. 
If $a_1a_4\in E(G)$ or $a_1$ has a neighbor of $G^*_S$, then $H+a_1a_4$ or the graph, obtained from $H\cup (G_S^*\cap G)\cup E_G(S,V(G_S^*))$ by contracting $V(G_S^*)$ to a single vertex, contains a subdivision of $K_{3,3}$, a contradiction again. As $\delta(G)\geq 3$, $\emptyset\neq N_G(a_1)\setminus \{b_1,b_2\}\subseteq \{a_2,a_3\}$. Then $\{a_2,a_3\}$ is a 2-vertex cut of $H\cup E_G(\{a_1\},\{a_2,a_3\})$, which is also a 2-vertex cut of $G^*$, a contradiction. 

\begin{figure}[h]
\centering
\includegraphics{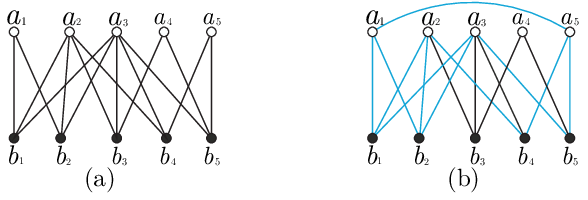}
\caption{\label{10} Illustration for the case that $|V(\widehat{H})|=4$ of Case 2 in the proof of Theorem \ref{main}.}
\end{figure}

{\textbf{Case 3.}} $|V(H)|\geq 12$.

If $\delta (H)\geq 3$, then $\Phi(H)\geq 12$ by Properties \ref{small}(i) and \ref{p}.  
So we assume that $\delta(H)=2$. If $|V(\widehat{H})|\geq 8$, then we obtain $\Phi(\widehat{H})\geq 12$ by Properties \ref{small} and \ref{p}. If $|V(\widehat{H})|=6$, then $|E(\widehat{H}_{\text{sim}})|\geq 6$ since $\widehat{H}_{\text{sim}}$ is matching covered. By Euler's formula, $|E(\widehat{H}_{\text{sim}})|\leq 2\times6-4=8$. Note that $\widehat{H}_{\text{sim}}$ is obtained by deleting the edges of a  matching of size 1, 2 or 3 from $K_{3,3}$. If $|E(\widehat{H}_{\text{sim}})|=6$, then $\widehat{H}_{\text{sim}}$ is a cycle. By Claim 1, each vertex in $V(\widehat{H}_{\text{sim}})\setminus S$ is obtained by the bicontraction of 2-degree vertices of $H$. By Claim 2, we have $\Phi(\widehat{H})\geq 2^3+2^3=16$ (see Fig. \ref{6-splitting}(a)). If $|E(\widehat{H}_{\text{sim}})|=7$, then $\Phi(\widehat{H})\geq 2\times2+2\times(2+2)=12$ by Claims 1 and 2  (see Fig. \ref{6-splitting}(b)). If $|E(\widehat{H}_{\text{sim}})|=8$, then $\Phi(\widehat{H})\geq 2\times(2+2)+(2+2)=12$ (see Fig. \ref{6-splitting}(c)). 

\begin{figure}
\centering
\includegraphics{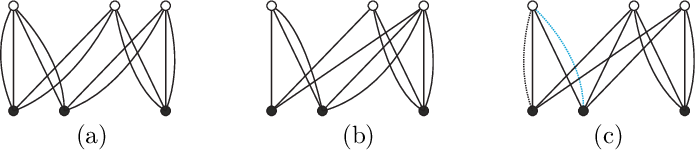}
\caption{\label{6-splitting} (a)\ $|(E(\widehat{H}_{\text{sim}})|=6$, (b)\ $|(E(\widehat{H}_{\text{sim}})|=7$, (c)\ $|(E(\widehat{H}_{\text{sim}})|=8$.}
\end{figure}

If $|V(\widehat{H})|=4$, then $H$ has at least four 2-degree vertices. Set $V(\widehat{H})=\{a_1,a_2,b_1',b_2'\}$. Assume that $|E_{\widehat{H}}(\{b_i'\},\{a_i\})|=x_i$ for $i=1,2$. Then $2\leq x_i\leq d_{\widehat{H}}(b_i')-2$ by Claim 2. 
If $b_i'$ is obtained by bicontracting at least two 2-degree vertices for $i=1,2$, then $d_{\widehat{H}}(b_i')\geq 5$ by Claim 1. So $\Phi(\widehat{H})\geq x_1x_2+(5-x_1)(5-x_2)$, where $2\leq x_1,x_2\leq 3$. A simple computation yields $\Phi(\widehat{H})\geq 12$. If $b_j'$ is obtained by bicontracting one 2-degree vertex for some $j\in \{1,2\}$, then $d_{\widehat{H}}(b_j')\geq 4$ and $d_{\widehat{H}}(b_{3-j}')\geq 6$ by Claim 1. So $\Phi(\widehat{H})\geq 2x_{3-j}+2\times(6-x_{3-j})=12$ by Claim 2, where $2\leq x_{3-j}\leq 4$. 

Consequently, we have $\Phi(H)=\Phi(\widehat{H})\geq 12$ by Property \ref{bicontract}(i), and the proof is completed.


\begin{thebibliography}{99}
\small \setlength{\itemsep}{-.2mm}
\bibitem{BM}J. A. Bondy, U. S. R. Murty, Graph Theory, Springer-Verlag, Berlin, 2008.

\bibitem{Berge}C. Berge, Sur le couplage maximum d'un graphe, C. R. Acad. Sci. Paris S\'{e}r. I Math. 247 (1958) 258-259.

\bibitem{CLM}M. H. de Carvalho, C. L. Lucchesi, U. S. R. Murty, On the number of perfect matchings in a bipartite graph, SIAM J. Discrete Math. 27 (2013) 940-958.

\bibitem{CLM14}M. H. de Carvalho, C. H. C. Little, Matching covered graphs with three removable classes, Electron. J. Combin. 21 (2014) \#P2.13.

\bibitem{CS}M. Chudnovsky, P. Seymour, Perfect matchings in planar cubic graphs. Combinatorica 32 (2012) 403-424.

\bibitem{EKKKN}L. Esperet, F. Kardo\v{s}, A. D. King, D. Kr\'al, S. L. Norine, Exponentially many perfect
matchings in cubic graphs, Adv. Math. 227 (2011) 1646-1664.

\bibitem{EPL}J. Edmonds, W. R. Pulleyblank, L. Lov\'asz, Brick decompositions and the matching rank of graphs, Combinatorica 2 (1982) 247-274.

\bibitem{GJVZ}J. Goedgebeur, J. Jooken, T. Van den Eede, C. T. Zamfirescu. On the number of perfect matchings in planar graphs, https://doi.org/10.48550/arXiv.2606.22253.

\bibitem{Hetyei}G. Hetyei, Rectangular configurations which can be covered by $2\times 1$ rectangles, P\'{e}csi Tan. F\H{o}isk. K\"{o}zl. 8 (1964) 351-367. (Hungarian)

\bibitem{HLZ}X. He, F. Lu, H. Zhang, Adjacent vertices of small degree in minimal matching covered graphs, https://doi.org/10.48550/arXiv.2604.00361. 

\bibitem{Kasteleyn}P. W. Kasteleyn, Graph theory and crystal physics, in: Graph Theory and Theoretical Physics, Academic Press, 1967, pp. 43-110.

\bibitem{L72}Lov\'{a}sz, On the structure of factorizable graphs, Acta Mat. Hungar. 23 (1972) 179-195.

\bibitem{LP09}L. Lov\'{a}sz, M. D. Plummer, Matching Theory, Ann. Discrete Math., Vol. 29, North-Holland, Amsterdam, 1986; AMS Chelsea Publishing, Providence, RI, 2009.

\bibitem{Zaks}J. Zaks, On the 1-factors of $n$-connected graphs, J. Combin. Theory Ser. B 11 (1971) 169-180.

\end{thebibliography}
\end{document}